\documentclass[a4paper]{amsart}

\usepackage{amssymb,amsthm,amsmath}
\swapnumbers

\usepackage[all]{xy}
\usepackage{hyperref}
\usepackage[shortlabels]{enumitem}
\usepackage{aliascnt}
\usepackage{microtype}

\newcommand{\CatCu}{\ensuremath{\mathrm{Cu}}}
\newcommand{\CuSgp}{$\CatCu$-sem\-i\-group}
\newcommand{\CuMor}{$\CatCu$-mor\-phism}
\newcommand{\axiomO}[1]{(O#1)}
\newcommand{\Cu}{\ensuremath{\mathrm{Cu}}}
\newcommand{\andSep}{\,\,\,\text{ and }\,\,\,}
\newcommand{\ca}{$C^*$-algebra}
\newcommand{\stHm}{${}^*$-homomorphism}

\newcommand{\AW}{$AW^*$}
\newcommand{\KK}{\mathcal{K}}

\newcommand{\NN}{\mathbb{N}}
\newcommand{\RR}{\mathbb{R}}
\newcommand{\CC}{\mathbb{C}}

\DeclareMathOperator{\QT}{QT}

\newtheorem{lma}{Lemma}[section]

\newaliascnt{thmCt}{lma}
\newtheorem{thm}[thmCt]{Theorem}
\aliascntresetthe{thmCt}

\newaliascnt{corCt}{lma}
\newtheorem{cor}[corCt]{Corollary}
\aliascntresetthe{corCt}

\newaliascnt{prpCt}{lma}
\newtheorem{prp}[prpCt]{Proposition}
\aliascntresetthe{prpCt}

\theoremstyle{definition}
\newtheorem{pgr}[lma]{}

\newaliascnt{dfnCt}{lma}
\newtheorem{dfn}[dfnCt]{Definition}
\aliascntresetthe{dfnCt}

\newaliascnt{rmkCt}{lma}
\newtheorem{rmk}[rmkCt]{Remark}
\aliascntresetthe{rmkCt}

\newaliascnt{qstCt}{lma}
\newtheorem{qst}[qstCt]{Question}
\aliascntresetthe{qstCt}

\newaliascnt{exaCt}{lma}

\aliascntresetthe{exaCt}

\numberwithin{equation}{section}

\title{Edwards' condition for quasitraces on C*-algebras}

\date{\today}
\author{Ramon Antoine}
\author{Francesc Perera}
\author{Leonel Robert}
\author{Hannes Thiel}

\address{
Ramon~Antoine \& Francesc~Perera, Departament de Matem\`{a}tiques,
Universitat Aut\`{o}noma de Barcelona,
08193 Bellaterra, Barcelona, Spain}
\email[]{ramon@mat.uab.cat, perera@mat.uab.cat}

\address{
Leonel~Robert,
Department of Mathematics,
University of Louisiana at Lafayette,
Lafayette, LA 70504-1010, USA}
\email{lrobert@louisiana.edu}

\address{
Hannes~Thiel, Mathematisches Institut, Universit\"at M\"unster,
Einsteinstr.~62, 48149 M\"unster, Germany}
\email[]{hannes.thiel@uni-muenster.de}

\begin{document}

\begin{abstract}
We prove that Cuntz semigroups of C*-algebras satisfy Edwards' condition with respect to every quasitrace.
This condition is a key ingredient in the study of the realization problem of functions on the cone of quasitraces as ranks of positive elements.
In the course of our investigation, we identify additional structure of the Cuntz semigroup of an arbitrary C*-algebra and of the cone of quasitraces.
\end{abstract}

\thanks{
The two first named authors were partially supported by MINECO (grant No.\  MTM2017-83487-P), and by the Comissionat per Universitats i Recerca de la Generalitat de Catalunya (grant No.\ 2017-SGR-1725).
The fourth named author was partially supported by the Deutsche Forschungsgemeinschaft (DFG, German Research Foundation) under the SFB 878 (Groups, Geometry \& Actions) and under Germany's Excellence Strategy EXC 2044-390685587 (Mathematics M\"{u}nster: Dynamics-Geometry-Structure).
}

\maketitle

\section{Introduction}

The rank of a positive element $a$ in a C*-algebra $A$ with respect to a trace $\tau$ (or, more generally, a quasitrace) is defined as $d_\tau(a)=\lim_n \tau(a^{1/n})$.
In case of a trace, this rank is nothing but the value of the support projection of $a$ in $A^{**}$ under the canonical extension of $\tau$ to a normal trace on $A^{**}$;
see \cite{OrtRorThi11CuOpenProj}.

If $A$ is unital and stably finite, then the set $\QT_1(A)$ of normalized quasitraces is a nonempty Choquet simplex.
Given an extreme quasitrace $\tau$ in $\QT_1(A)$, it was shown in \cite[Theorem~4.7]{Thi17arX:RksOps} that for any two positive elements $a$ and $b$ in $A$, 
the minimum of the ranks of $a$ and $b$ with respect to $\tau$ can be approximated by the ranks of positive elements $c$ that are dominated by $a$ and $b$ in the sense of Cuntz:
\[
\min\big\{ d_\tau(a), d_\tau(b) \big\}
= \sup\big\{ d_\tau(c) : c\precsim a,b \big\}.
\]
This property was termed \emph{Edwards' condition for $\tau$} by the fourth named author due to its relation with the work in \cite{Edw69UniformApproxAff}.
This paper concerns the extension of Edwards' condition to all quasitraces (not necessarily extremal) defined on a general (not necessarily unital) C*-algebra.

Edwards' condition for extremal, normalized quasitraces was a crucial ingredient in  \cite{Thi17arX:RksOps} for the solution of the rank problem for unital, simple C*-algebras of stable rank one.
In the same spirit, the general Edwards' condition as developed in this paper is a crucial ingredient in \cite{AntPerRobThi18arX:CuntzSR1} for the solution of the rank problem for general C*-algebras of stable rank one.

The \emph{rank problem} for a C*-algebra $A$ is to determine which functions on the topological cone $\QT(A)$ of quasitraces on $A$ arise as the ranks of positive operators in $A$.
Here, the \emph{rank} of $a$ in $A_+$ is the function that associates to each quasitrace $\tau$ the rank of $a$ with respect to $\tau$; 
see \cite{DadTom10Ranks}, \cite{Thi17arX:RksOps}, \cite{AntPerRobThi18arX:CuntzSR1}. 
The rank problem for $A$ is closely connected to the question of whether the set of ranks of elements in $A_+$ is closed under infima, that is, if $f,g\colon\QT(A)\to[0,\infty]$ are realized as the ranks of positive elements in $A$, is the same true for $f\wedge g$?
Loosely speaking, Edward's condition is the requirement that this can at least be done pointwise, that is, \emph{the infimum of the ranks of two positive elements $a$ and $b$ can be pointwise approximated by the ranks of elements dominated by $a$ and $b$};
see \autoref{dfn:Edwards}.

The Cuntz semigroup of a C*-algebra as introduced in \cite{CowEllIva08CuInv} satisfies a series of properties (see below for details) denoted \axiomO{1}-\axiomO{6}.
In \autoref{sec:Cu}, we show that Cuntz semigroups satisfy a new property, which we call \axiomO{7};
see \autoref{dfn:O7} and \autoref{prp:O7}.
This property allows us to deal naturally with ideals in the semigroup.
In particular, it allows us to obtain infima of elements in the Cuntz semigroup under the assumption that one of the elements is idempotent.
Note that idempotents in the Cuntz semigroup of a separable C*-algebra are in natural correspondence with the lattice of closed, two-sided ideals.

Given a C*-algebra $A$, the Cuntz semigroup $\Cu(A)$ appears naturally in the study of quasitraces on $A$.
Building on results from \cite{Cun78DimFct} and \cite{BlaHan82DimFct}, it was shown in \cite{EllRobSan11Cone} that the cone $\QT(A)$ of $[0,\infty]$-valued $2$-quasitraces on $A$ is homeomorphic to the cone $F(\Cu(A))$ of functionals on the Cuntz semigroup of $A$.
Therefore, ranks are naturally viewed as elements in the dual of this cone (or in the second dual of the semigroup), and hence their properties may be obtained from the study of structural properties of both $\Cu(A)$ and its cone of functionals.
This will be done, respectively, in Sections~2 and~3.

In \autoref{sec:cones} we study compact cones and their duals.
We apply our results to cones of functionals of Cuntz semigroups, continuing the work in \cite{EllRobSan11Cone} and \cite{Rob13Cone}.
In particular, we prove that the cone of quasitraces on a C*-algebra $A$ satisfies Riesz refinement, \autoref{prp:coneSemilatticeOrdered}, a result which is significant towards establishing Edwards' condition for Cuntz semigroups of C*-algebras.

\autoref{sec:Edwards} exclusively concerns Edwards' condition in an abstract setting.
As noted above, ranks arise as elements in the dual of a suitable monoid satisfying Riesz refinement.
We use the Riesz-Kantorovich type description of infima in this setting (see \eqref{prp:MdualInj:eqInfDual}) to generalize Edwards's condition for arbitrary functionals.
The main result of the section states that, for semigroups satisfying \axiomO{7}, one can verify Edwards' condition on functionals taking finite values.

Finally, in \autoref{sec:ca} we show that Cuntz semigroups of \ca{s} satisfy Edwards' condition;
see \autoref{prp:CuA-Edwards}.
Our method of proof follows the line of attack developed by the fourth author in \cite{Thi17arX:RksOps} combined with the results obtained in the previous sections.

\section{Properties of Cuntz semigroups}
\label{sec:Cu}

Let $A$ be a C*-algebra.
In \cite{Cun78DimFct}, Cuntz introduced the following relations for positive elements $a,b\in A$:
$a\precsim b$ if there is a sequence $(x_n)_n$ in $A$ such that $\lim_n \| a - x_nbx_n^*\|=0$;
$a\sim b$ provided that $a\precsim b$ and $b\precsim a$.

The Cuntz semigroup of $A$ is defined as $\Cu(A):=(A\otimes \mathcal K)_+/\!\sim$, where $\mathcal{K}$ denotes the algebra of compact operators on $\ell^2(\NN)$.
We denote the class of a positive element $a\in A\otimes \mathcal K$ by $[a]$.
Then $\Cu(A)$ becomes an ordered abelian semigroup with order induced by the subequivalence $\precsim$ and addition induced by $[a]+[b]=[\left(\smallmatrix a & 0 \\ 0 & b\endsmallmatrix\right)]$.

\subsection{Properties \axiomO{1}-\axiomO{6}}

The Cuntz semigroup of a C*-algebra is known to satisfy a number of order properties, which we now briefly recall.
The reader is referred to \cite{CowEllIva08CuInv} and \cite{AntPerThi18TensorProdCu} for background on the Cuntz semigroups of C*-algebras and their abstract counterparts, \CuSgp{s}.

Let $S$ be an ordered set such that every increasing sequence has a supremum.
Given $x,y\in S$ we say that $x$ is way-below $y$ if whenever $y\leq \sup_n y_n$ for some increasing sequence $(y_n)_n$, then there exists $n_0\in\NN$ such that $x\leq y_{n_0}$.
We denote this relation by $x\ll y$.

Suppose now that $S$ is a positively ordered monoid, that is, an ordered monoid such that $0\leq x$ for all $x\in S$.
Consider the following properties on $S$:
\begin{enumerate}
\item[\axiomO{1}]
Every increasing sequence in $S$ has a supremum.
\item[\axiomO{2}]
For each $x\in S$ there exists an $\ll$-increasing sequence $(x_n)_n$ such that $x=\sup_n x_n$.
\item[\axiomO{3}]
If $x_1\ll y_1$ and $x_2\ll y_2$ then $x_1+x_2\ll y_1+y_2$.
\item[\axiomO{4}]
If $(x_n)_n$ and $(y_n)_n$ are increasing sequences then $\sup_n (x_n+y_n)=\sup_n x_n+\sup_n y_n$
\end{enumerate}	
By a \emph{\CuSgp} we understand a positively ordered monoid satisfying \axiomO{1}-\axiomO{4}.
A map between \CuSgp{s} is called a \emph{\CuMor} if it is a monoid homomorphism that preserves order, suprema of increasing sequences, and the way-below relation.
It was shown in \cite{CowEllIva08CuInv} that the Cuntz semigroup of a C*-algebra is a \CuSgp, and that a \stHm{} $A\to B$ naturally induces a \CuMor{} $\Cu(A)\to\Cu(B)$.

The Cuntz semigroup of a C*-algebra also satisfies the following two properties:
\begin{enumerate}
\item[\axiomO{5}]
For all $x'\ll x\leq y$ and $w'\ll w$ such that $x+w\leq y$ there exists $z$ such that $x'+z\leq y\leq x+z$ and $w'\ll z$.
\item[\axiomO{6}]
For all $x'\ll x\leq y+z$ there exist $y',z'$ such that $x'\leq y'+z'$, $y'\leq x,y$ and $z'\leq x,z$.
\end{enumerate}
That Cuntz semigroups of C*-algebras satisfy \axiomO{5} was proved in \cite[Proposition~4.6, p.34]{AntPerThi18TensorProdCu}.
We will often use a weaker version of \axiomO{5} that first appeared in \cite{RorWin10ZRevisited}:
For all $x'\ll x\leq y$ there exists $z$ such that $x'+z\leq y\leq x+z$.
It was shown in \cite{Rob13Cone} that Cuntz semigroups of C*-algebras satisfy \axiomO{6}.

\subsection{Property \axiomO{7}}

We identify a new property that Cuntz semigroups of C*-algebras satisfy. 

\begin{dfn}
\label{dfn:O7}
A \CuSgp{} $S$ is said to satisfy \axiomO{7} if for all $x_1',x_1,x_2',x_2,w\in S$ satisfying
\[
x_1'\ll x_1\leq w \andSep x_2'\ll x_2\leq w,
\]
there exists $x\in S$ such that $x_1',x_2'\ll x\leq w,x_1+x_2$.
\end{dfn}

\begin{prp}
\label{prp:O7}
The Cuntz semigroup of every C*-algebra satisfies \axiomO{7}.
\end{prp}
\begin{proof}
Let $A$ be a C*-algebra. Let $x_i$, $x_i'$, $w\in \Cu(A)$ satisfy $x_i'\ll x_i\leq w$ for $i=1,2$. Choose positive elements $b_1,b_2,a\in A\otimes\mathcal K$ such that $x_1=[b_1]$, $x_2=[b_2]$, and $w=[a]$.

Since  $x_1'\ll [b_1]$ and $x_2'\ll [b_2]$, we may choose $\epsilon>0$ such that 
$x_i'\leq [(b_i-\epsilon)_+]$ for $i=1,2$. Since $[b_i]\leq [a]$, there are positive elements $c_1,c_2\in \overline{a(A\otimes\mathcal K)a}$ such that $(b_i-\epsilon)_+\sim c_i$ for $i=1,2$.

Set $x=[c_1+c_2]$.
Then 
\[
x_i'\ll [(b_i-\epsilon)_+]=[c_i]\leq [c_1+c_2]=x\hbox{ for }i=1,2.
\]
Also, since $c_1+c_2$ belongs to $\overline{a(A\otimes\mathcal K)a}$, we have $x\leq [a]=w$.
Using at the first step that the Cuntz class of the sum of two positive elements is always dominated by the sum of their Cuntz classes (\cite[Lemma~2.10]{AraPerTom11Cu}), we obtain
\[
x\leq [c_1]+[c_2]=[(b_1-\epsilon)_+]+[(b_2-\epsilon)_+]\leq x_1+x_2,
\]
as desired.
\end{proof}

An \emph{ideal} of a \CuSgp{} is a downward hereditary subsemigroup closed under suprema of increasing sequences.
(See \cite[Section~5.1, p.37ff]{AntPerThi18TensorProdCu} for more details.)
The relevance of \axiomO{7} when dealing with ideals of a \CuSgp{} is demonstrated in the following result.

\begin{prp}
\label{prp:infWithIdealDirected}
Let $S$ be a \CuSgp{} satisfying \axiomO{7}, let $w\in S$, and let $J\subseteq S$ be an ideal.
Then the set $\{x\in S : x\in J, x\leq w \}$ is upward directed.
\end{prp}
\begin{proof}
Notice that $\{x\in S : x\leq w, x\in J\}$ is a downward hereditary subset closed under suprema of increasing sequences.
Thus, by \cite[Lemma~3.2]{AntPerRobThi18arX:CuntzSR1}, it suffices to show that
\[
\big\{ x'\in S : \text{there exists }x\in J\text{ such that }x'\ll x,\, x\leq w \big\}
\]
is upward directed.
To this end, suppose that $x'$, $y'\in S$ satisfy that $x'\ll x$, $y'\ll y$ for some elements $x$, $y\in J$ such that  $x,y\leq w$. We deduce by \axiomO{7} that there exists $z\in S$ such that $x',y'\ll z\leq w,x+y$. Since $z\leq x+y$, and since $J$ is an ideal, we have $z\in J$. Choose $z'\in S$ with $z'\ll z$, and such that $x'\leq z'$ and $y'\leq z'$. Then $z'$ is in the set displayed above and, being an upper bound for both $x'$ and $y'$, this shows that this set is upward directed, as desired.
\end{proof}

A \CuSgp{} is called \emph{countably based} if it contains a countable subset such that
every element is the supremum of a $\ll$-increasing sequence with terms in the said countable subset.
It is a standard result that in a countably based \CuSgp{} every directed subset admits a supremum;
see \cite[Remarks~3.1.3, p.21f]{AntPerThi18TensorProdCu}.
Cuntz semigroups of separable C*-algebras are countably based (see, for example, \cite[Lemma 1.3]{AntPerSan11:Pullbacks}).

Let $J$ be an ideal of a countably based \CuSgp{} $S$.
Since ideals of \CuSgp{s} are upward directed, $J$ has a largest element $w_J:=\sup J$ (see also \cite[Paragraph~5.1.6, p.39f]{AntPerThi18TensorProdCu}).
Further, this element is idempotent, that is, $2w_J=w_J$.
Conversely, given an idempotent $w\in S$, the order ideal generated by $w$ is an ideal of $S$ with supremum $w$.
In light of this correspondence, \autoref{prp:infWithIdealDirected} immediately implies the following result. 

\begin{thm}
\label{prp:infWithIdeal}
Let $S$ be a countably based \CuSgp{} satisfying \axiomO{7}.
Then each $x\in S$ and each idempotent element $w\in S$ have an infimum $x\wedge w$ in $S$.
\end{thm}

Moreover, these infima with idempotent elements are well behaved as the following results illustrates.

\begin{thm}
\label{prp:infIdealGenCu}
Let $S$ be a countably based \CuSgp{} satisfying \axiomO{5}-\axiomO{7}.
Let $w\in S$ be an idempotent element.
Then the following are satisfied:
\begin{enumerate}[{\rm (i)}]
\item
The map $S\to S$ given by $x\mapsto x\wedge w$ is a monoid homomorphism preserving the order and the suprema of increasing sequences.
\item
Given $x,y\in S$, we have $x\leq y+w$ if and only if $x+(y\wedge w) \leq y+(x\wedge w)$. 
\item
We have 
\[
x\wedge w_1+x\wedge w_2 = x\wedge(w_1\wedge  w_2) + x\wedge(w_1+w_2)
\]
for all $x\in S$ and idempotents $w_1,w_2\in S$.
\end{enumerate}
\end{thm}
\begin{proof}
(i):
Define $\varrho_w\colon S\to S$ by $\varrho_w(x):=x\wedge w$.
It is obvious that $\varrho_w$ is order preserving.
To prove additivity, let $x,y\in S$.
Since $x\wedge w+y\wedge w\leq x+y$ and $x\wedge w+y\wedge w\leq 2w=w$, we have 
\[
\varrho_w(x)+\varrho_w(y)
= x\wedge w+y\wedge w
\leq (x+y)\wedge w
= \varrho_w(x+y).
\]
To show the converse inequality, set $z=(x+y)\wedge w$, and let $z'\ll z$.
Apply \axiomO{6} for $z'\ll z\leq x+y$ to obtain $x',y'\in S$ satisfying
\[
z' \leq x'+y',\quad
x'\leq x,z, \andSep 
y'\leq y,z.
\]
Since $x'\leq z\leq w$ and $x'\leq x$, we have $x'\leq x\wedge w$.
Analogously, we deduce that $y'\leq y\wedge w$.
Hence, $z'\leq x\wedge w+y\wedge w$.
Since this holds for all $z'\ll z$, we obtain 
\[
\varrho_w(x+y)
= z\leq x\wedge w +y\wedge w
= \varrho_w(x)+\varrho_w(y).
\] 

Finally, let us show that $\varrho_w$ preserves sequential suprema.
Let $(x_n)_{n=1}^\infty$ be an increasing sequence in $S$.
The inequality 
\[
\sup_n \varrho_w(x_n)
= \sup_n (x_n\wedge w)
\leq (\sup_n x_n)\wedge w
= \varrho_w(\sup_n x_n)
\]
is clear.
Set $z=(\sup_n x_n)\wedge w$ and let $z'\ll z$.
Since $z\leq \sup_n x_n$, there exists $n$ such that $z'\leq x_n$.
Also, $z'\leq z\leq w$. 
Therefore
\[
z'\leq x_n\wedge w\leq \sup_n (x_n\wedge w).
\]
Since this holds for all $z'\ll z$, we obtain
\[
\varrho_w(\sup_n x_n)= z\leq \sup_n (x_n\wedge w)=\sup\varrho_w(x_n).
\]

(ii):
Let $x,y\in S$.
If $x+(y\wedge w) \leq y+(x\wedge w)$, then 
\[
x \leq x+(y\wedge w)
\leq y+(x\wedge w)
\leq y+w.
\]

To show the converse implication, assume that $x\leq y+w$, and let $y'\ll y\wedge w$.
By \axiomO{5}, we can choose $z$ such that $y'+z\leq y\leq y\wedge w+z$.
Then $x\leq y\wedge w+z+w=z+w$. Let $x'\ll x$.
By \axiomO{6}, $x'\leq z+x\wedge w$.
Adding $y'$ on both sides we get $x'+y'\leq y+x\wedge w$.
Passing to the supremum over all $x'\ll x$ and $y'\ll y\wedge w$, the result follows.

(iii):
Let $x\in S$ and let $w_1,w_2\in S$ be idempotents.
By (i), $w_1\wedge w_2$ is also an idempotent.
It thus makes sense to write $x\wedge(w_1\wedge w_2)$ and this agrees with $(x\wedge w_1)\wedge w_2$.
We first show that
\[
x\wedge(w_1+w_2)+w_1 = x\wedge w_2 + w_1.
\]
The inequality `$\geq$' is clear.
On the other hand, applying (ii) to $x\wedge (w_1+w_2)\leq w_1+w_2$ at the first step yields
\[ 
x\wedge(w_1+w_2)+w_1\wedge w_2\leq w_1+x\wedge (w_1+w_2)\wedge w_2 =w_1+x\wedge w_2.  
\]
Adding $w_1$ to the previous inequality, we obtain the desired reverse inequality
\[ 
x\wedge(w_1+w_2)+w_1 \leq x\wedge w_2 + w_1.
\]

Given $y,z\in S$ satisfying $y+w_1 = z+w_1$, it follows from (ii) that $y+z\wedge w_1=z+y\wedge w_1$.
Applying this for $y=x\wedge(w_1+w_2)$ and $z=x\wedge w_2$, we get
\[
x\wedge(w_1+w_2)+(x\wedge w_2)\wedge w_1=x\wedge w_2+(x\wedge (w_1+w_2))\wedge w_1,
\]
which implies the desired equality.
\end{proof}

\begin{rmk}
Let $A$ be a \ca{}, and let $J$ be a $\sigma$-unital, closed, two-sided ideal.
Then $\Cu(A)$ is a \CuSgp{} satisfying \axiomO{5}-\axiomO{7}.
We identify $\Cu(J)$ with the ideal $\{[a]\in\Cu(A) : a\in (J\otimes\KK)_+\}$ of $\Cu(A)$.
Since $J$ is $\sigma$-unital, there exists a largest element in $\Cu(J)$, denoted $w_J$.

Recall that $\Cu(A)$ can be identified with certain equivalence classes of countably generated, right Hilbert C*-modules over $A$;
see \cite{CowEllIva08CuInv}, see also \cite{AraPerTom11Cu}.
If $M$ is a countably generated, right Hilbert C*-module over $A$, then $MJ$ is a countably generated, right Hilbert C*-module over $J$, and $[MJ]$ - the class of $MJ$ in $\Cu(J)$ - depends only on the class of $M$, which is the justification to denote $[MJ]$ by $[M]J$;
see \cite{CiuRobSan10CuIdealsQuot}.
One can show that
\[
[M]\wedge w_J = [M]J
\]
in $\Cu(A)$.
Hence, \autoref{prp:infIdealGenCu}(i) and~(ii) generalize (and recover) Proposition~4.3 and Theorem~1.1 in \cite{CiuRobSan10CuIdealsQuot} in the case that $A$ is a separable C*-algebra.
\end{rmk}

If $S$ is the Cuntz semigroup of a (not necessarily separable) C*-algebra, then the results in \cite{CiuRobSan10CuIdealsQuot} show that the infimum of any $x\in S$ and any idempotent $w\in S$ exist.
Thus, for Cuntz semigroups of C*-algebras, \autoref{prp:infWithIdeal} holds without the assumption of countable generation.
It seems unclear if the same holds for \CuSgp{s}:

\begin{qst}
Let $S$ be a \CuSgp{} satisfying \axiomO{7}.
Do each $x\in S$ and each idempotent $w\in S$ admit an infimum in $S$?
What, if we additionally assume that $S$ satisfies \axiomO{5} and \axiomO{6}?
\end{qst}

\section{Cones and their duals}
\label{sec:cones}

Here we establish a number of results on algebraically ordered, compact cones and their duals.
We then apply these results to our main object of study: the cone~$F(S)$ of functionals on a \CuSgp{} $S$. 

\subsection{Algebraically ordered compact cones}

Recall that a \emph{cone} is a commutative monoid~$C$ together with a scalar multiplication by $(0,\infty)$. More specifically, the scalar multiplication is a map $(0,\infty)\times C\to C$, denoted $(t,a)\mapsto ta$, that is additive in each variable, and such that $(st)a=s(ta)$ and $1a=a$, for all $s,t\in(0,\infty)$ and $a\in C$.
Note that we do not define scalar multiplication by $0$.
A \emph{topological cone} is a cone together with a topology such that addition and scalar multiplication are jointly continuous.
(Here we equip $(0,\infty)$ with the usual Hausdorff topology of real numbers.)

The algebraic pre-order on a cone $C$ is defined as $a\leq b$ if $a+c=b$ for some $c\in C$.
If the algebraic pre-order is an order then we speak of an \emph{algebraically ordered cone}.

The following result is standard.
It holds more generally in compact, ordered spaces as studied by Nachbin, \cite{Nach65TopOrder}, see \cite[Proposition~VI-1.3, p.441]{GieHof+03Domains}.

\begin{prp}
\label{prp:cpctConeDcpo}
Let $C$ be an algebraically ordered, compact cone.
Then $C$ is both directed complete and filtered complete. Moreover, given an upward (downward) directed subset $D$ of $C$, considering $D$ as a net indexed over itself, $D$ converges to~$\sup D$ (to $\inf D$).
\end{prp}

Let $C$ be an algebraically ordered, compact cone.
We set
\[
E(C) := \big\{ a\in C : 2a=a \big\}.
\]
Given $a\in C$, the sequence $(\tfrac{1}{n}a)_n$ is decreasing and therefore converges. We have $2\lim_n\tfrac{1}{n} a=\lim_n\tfrac{1}{n} a$, which justifies to define $\varepsilon\colon C\to E(C)$ by
\[
\varepsilon(a) := \lim_n \tfrac{1}{n} a,
\]
for $a\in C$.
It is straightforward to verify that $\varepsilon$ is additive and order-preserving. Moreover, we have $\varepsilon(a)+a=a$ for every $ a\in C$.

Following Wehrung (Definitions~1.12, 2.10, and~3.1 in \cite{Weh92InjectivePOM1}), we say that $C$ is \emph{pseudo-cancellative} if for all $a,b,c\in C$ with $a+c\leq b+c$ there exists $d\in C$ such that $a\leq b+d$ and $d+c=c$.

\begin{lma}
\label{prp:cpctConePseudoCanc}
Let $C$ be an algebraically ordered, compact cone.
Let $a,b,c\in C$ satisfy $a+c\leq b+c$.
Then $a+\varepsilon(c) \leq b+\varepsilon(c)$.
In particular, $C$ is pseudo-cancellative.
\end{lma}
\begin{proof}
Multiplying by $\tfrac{1}{2}$ in $a+c\leq b+c$ we get
\[
\tfrac{1}{2}a + \tfrac{1}{2}c \leq \tfrac{1}{2}b + \tfrac{1}{2}c.
\]
Adding $\tfrac{1}{2}a$, and then using the above inequality, we obtain
\[
a + \tfrac{1}{2}c
\leq \tfrac{1}{2}a + \tfrac{1}{2}b + \tfrac{1}{2}c
\leq b + \tfrac{1}{2}c.
\]
Thus, we inductively deduce that $a + \tfrac{1}{2^n}c \leq b + \tfrac{1}{2^n}c$, for each $n\in\NN$.
It follows that
\[
a+\varepsilon(c)
= \lim_n (a+\tfrac{1}{2^n}c)
\leq \lim_n (b+\tfrac{1}{2^n}c)
= b+\varepsilon(c),
\]
as desired.
\end{proof}

Recall that a monoid $M$ is said to satisfy \emph{Riesz refinement} if for all $a_1,a_2,b_1,b_2\in M$ with $a_1+a_2=b_1+b_2$ there exist $x_{i,j}\in M$, for $i,j=1,2$, such that $a_i=x_{i,1}+x_{i,2}$ for $i=1,2$, and $b_j=x_{1,j}+x_{2,j}$ for $j=1,2$.

An \emph{inf-semilattice ordered monoid} is a positively ordered monoid $M$ that is an inf-semilattice and such that addition is distributive over $\wedge$, that is,
\begin{align}
\label{pgr:pom:eqDistrWedge}
a+(b\wedge c)=(a+b)\wedge(a+c),
\end{align}
for all $a,b,c\in M$.
Dually, one defines sup-semilattice ordered monoids.
A \emph{lattice-ordered monoid} is a positively ordered monoid $M$ that is a lattice and such that addition is distributive over $\wedge$ and $\vee$, that is, for all $a,b,c\in M$ we have \eqref{pgr:pom:eqDistrWedge} and $a+(b\vee c)=(a+b)\vee(a+c)$.

The following proposition is a consequence of results of Wehrung:	

\begin{prp}
\label{prp:coneSemilatticeOrdered}
Let $C$ be an algebraically ordered, compact cone.
Then the following are equivalent:
\begin{enumerate}
\item
$C$ satisfies Riesz refinement.
\item
$C$ is inf-semilattice ordered.
\item
$C$ is lattice ordered.
\end{enumerate}
\end{prp}
\begin{proof}
By \autoref{prp:cpctConePseudoCanc}, $C$ is pseudo-cancellative.
Therefore, it follows from \cite[Proposition~1.23]{Weh92InjectivePOM1} that~(2) implies~(1);
and it follows from \cite[Lemma~1.16]{Weh92InjectivePOM1} that (1) implies~(3).
\end{proof}

Let $C$ be an algebraically ordered, compact cone satisfying Riesz refinement.
Let~$C^*$ denote the collection of linear maps $C\to[0,\infty]$, where by a linear map we understand an additive map satisfying $f(0)=0$, and such that $f(ta)=tf(a)$ for all $t\in(0,\infty)$ and $a\in C$.
We equip $C^*$ with pointwise addition and the algebraic order.
In fact, the algebraic order on $C^*$ agrees with the pointwise order.
The property of Riesz refinement of $C$ implies that $C^*$ is lattice ordered.
Further, the infimum and supremum of elements $f,g\in C^*$ are given by the Riesz-Kantorovich formulas:
\begin{align}
	\label{prp:MdualInj:eqInfDual}
	(f\wedge g)(a)
	&= \inf\big\{ f(a_1)+g(a_2) : a=a_1+a_2 \big\}, \\
	\label{prp:MdualInj:eqSupDual}
	(f\vee g)(a)
	&= \sup\big\{ f(a_1)+g(a_2) : a=a_1+a_2 \big\},
\end{align}
for $a\in C$.
(See \cite[Lemma~1.12]{Sho90DualityCardAlg}.)

A map $f\in C^*$ is said to be \emph{lower semicontinuous} if for every $t\in[0,\infty)$ the set $\{a\in C : f(a)\leq t\}$ is closed (in the topology of $C$).
We let $C'$ denote the family of lower semicontinuous maps in $C^*$.
It is easy to see that $C'$ is closed under addition.
The partial order on $C'$ (induced by $C^*$) is the pointwise order, and it is usually not the algebraic order, even though $C^*$ is algebraically ordered.

\begin{lma}
\label{prp:coneDualRealizeInfSup}
Let $C$ be an algebraically ordered, inf-semilattice ordered, compact cone, let $f,g\in C'$, and let $a\in C$.
Then the infimum in \eqref{prp:MdualInj:eqInfDual} is realized.
More precisely, there exist $a_1,a_2\in C$ with $a=a_1+a_2$ and $(f\wedge g)(a)=f(a_1)+g(a_2)$.
\end{lma}
\begin{proof}
Choose sequences $(a_{1,n})_n$ and $(a_{2,n})_n$ in $C$ such that $a=a_{1,n}+a_{2,n}$ for each~$n$, and such that
\[
(f\wedge g)(a)=\lim_n \big( f(a_{1,n})+g(a_{2,n}) \big).
\]

Since $C$ is compact, we can choose convergent subnets
such that $(a_{1,n(j)})_{j\in J}$ and $(a_{2,n(j)})_{j\in J}$ converge to some $a_1$ and $a_2$ in $C$, respectively.
Then $a=a_1+a_2$.
Using this at the last step, and using that $f$ and $g$ are lower semicontinuous at the second step, we obtain
\[
(f\wedge g)(a)
= \lim_j \big( f(a_{1,n(j)})+g(a_{2,n(j)}) \big)
\geq f(a_1)+g(a_2)
\geq (f\wedge g)(a).
\]

Thus, we have $a_1,a_2\in C$ such that $(f\wedge g)(a) =f(a_1)+g(a_2)$ and $a=a_1+a_2$.
\end{proof}

The following result contains analogs of results in \cite{Weh92InjectivePOM1}
for lower semicontinuous functionals.
It can also be considered as an analog of \cite[Theorem~4.2.2]{Rob13Cone}.

\begin{thm}
\label{prp:coneLscDual}
Let $C$ be an algebraically ordered, inf-semilattice ordered, compact cone.
Then $C'\subseteq C^*$ is closed under finite infima and directed suprema.
Moreover, given $f,g,h\in C'$ and an increasing net $(g_j)_j$ in $C'$, we have
\begin{align}
\label{prp:coneLscDual:eqInfDirectSup}
f\wedge ( \sup_j g_j ) &= \sup_j (f\wedge g_j), \\
\label{prp:coneLscDual:eqInfSum}
f+(g\wedge h) &= (f+g)\wedge(f+h).
\end{align}
\end{thm}
\begin{proof}
By \autoref{prp:coneSemilatticeOrdered}, $C$ satisfies Riesz refinement, and thus we obtain that $C^*$ is lattice-ordered with infimum given by \eqref{prp:MdualInj:eqInfDual}.

We first show that $C'$ is closed under infima. Let $f,g\in C'$. In order to verify that $f\wedge g$ in $C^*$ is lower semicontinuous (and thus it belongs to $C'$), we have to check that the set
\[
T:=\{a\in C \colon (f\wedge g)(a)\leq t\}
\]
is closed for any $t\in [0,\infty)$.
Let $(a_j)_{j\in J}$ be a net in $T$ that converges to $a$ in $C$.
For each $j\in J$ apply \autoref{prp:coneDualRealizeInfSup} to obtain $a_{j,1},a_{j,2}\in C$ such that
\[
f(a_{j,1})+g(a_{j,2}) = (f\wedge g)(a_j)\leq t,\andSep
a_j=a_{j,1}+a_{j,2}.
\]
Using that $C$ is compact, choose a subnet $(j(i))_{i\in I}$ such that $(a_{j(i),1})_{i\in I}$ and $(a_{j(i),2})_{i\in I}$ converge to some $a_1$ and $a_2$ in $C$, respectively.
Then $a=a_1+a_2$.
Using at the third step that $f$ and $g$ are lower semicontinuous, we deduce
\begin{align*}
(f\wedge g)(a)
&\leq f(a_1)+g(a_2)\\
&= f\big( \lim_{i\in I} a_{j(i),1} \big) + g\big( \lim_{i\in I} a_{j(i),2} \big) \\
&\leq \liminf_{i\in I} f\big( a_{j(i),1} \big) + \liminf_{i\in I} g\big( a_{j(i),2} \big) \\
&\leq \liminf_{i\in I} \left( f(a_{j(i),1}) + g(a_{j(i),2}) \right)
\leq t.
\end{align*}

Secondly, it is straightforward to verify that lower semicontinuity passes to suprema of upward directed families. Further, \eqref{prp:coneLscDual:eqInfSum} follows using that $C^*$ is lattice-ordered.

Finally, let us verify \eqref{prp:coneLscDual:eqInfDirectSup}.
Let $f\in C'$, and let $(g_j)_{j\in J}$ be an increasing net in~$C'$.
Set $g:=\sup_j g_j$.
It is straightforward to verify that $f\wedge g \geq \sup_j(f\wedge g_j)$.
To show the converse inequality, let $a\in C$.
Given $j\in J$, apply \autoref{prp:coneDualRealizeInfSup} to obtain $a_{j,1},a_{j,2}\in C$ such that
\[
(f\wedge g_j)(a)=f(a_{j,1})+g_j(a_{j,2}),\andSep a=a_{j,1}+a_{j,2}.
\]
Using that $C$ is compact, choose a subnet $(j(i))_{i\in I}$ such that $(a_{j(i),1})_{i\in I}$ and $(a_{j(i),2})_{i\in I}$ converge to some $a_1$ and $a_2$ in $C$, respectively.
Then $a=a_1+a_2$.
For each $i_0\in I$, using at the first step that the net $(g_j)_j$ is increasing, and using at the second step that $g_{j(i_0)}$ is lower semicontinuous, we obtain
\[
\lim_i g_{j(i)}(a_{j(i),2})
\geq \lim_i g_{j(i_0)}(a_{j(i),2})
\geq g_{j(i_0)}(a_2).
\]
Since this holds for all $i_0\in I$, we deduce
\[
\lim_{i\in I} g_{j(i)}(a_{j(i),2})
\geq \sup_{i_0\in I} g_{j(i_0)}(a_2)
= g(a_2).
\]
Using this inequality and using that $f$ is lower semicontinuous at the third step, we obtain
\begin{align*}
	\sup_{j\in J}(f\wedge g_j)(a)
	&= \sup_{j\in J} \big( f(a_{j,1})+g_j(a_{j,2}) \big) \\
	&= \lim_{i\in I} \big( f(a_{j(i),1})+g_{j(i)}(a_{j(i),2}) \big) \\
	&\geq f(a_1)+g(a_2) \\
	&\geq (f\wedge g)(a),
\end{align*}
as desired.
\end{proof}

\subsection{The cone of functionals on a Cu-semigroup}
\label{coneFS}

Let $S$ be a \CuSgp{}. 
A map $\lambda\colon S\to [0,\infty]$ is called a \emph{functional} if $\lambda(0)=0$ and if $\lambda$ preserves addition, order and suprema of increasing sequences.
We denote by $F(S)$ the set of functionals on $S$.
This is a cone when endowed with the operations of pointwise addition and pointwise scalar multiplication by positive reals.
We also equip $F(S)$ with the topology such that $\lambda_j\to \lambda$, for a given net $(\lambda_j)_j$ and a functional $\lambda$ in $F(S)$, provided that 
\[
\limsup_j \lambda_j(x') 
\leq \lambda(x) 
\leq \liminf_j \lambda_j(x)
\]
for all $x',x\in S$ with $x'\ll x$.
Then $F(S)$ is a compact cone;
see \cite{EllRobSan11Cone, Rob13Cone},
see also \cite[Theorem~3.17]{Kei17CuSgpDomainThy}.
If $S$ satisfies \axiomO{5}, then $F(S)$ is algebraically ordered;
see \cite[Proposition~2.2.3]{Rob13Cone}.
Further, if $S$ satisfies \axiomO{5} and \axiomO{6}, then $F(S)$ is an algebraically ordered, lattice ordered, compact cone;
see \cite[Theorem~4.1.2]{Rob13Cone}.
Combined with \autoref{prp:coneSemilatticeOrdered}, we deduce the following:

\begin{thm}
\label{prp:refinementFS}
Let $S$ be a \CuSgp{} satisfying \axiomO{5} and \axiomO{6}.
Then $F(S)$ satisfies Riesz refinement.
\end{thm}

Since the Cuntz semigroup of a C*-algebra is a \CuSgp{} satisfying \axiomO{5} and \axiomO{6}, the previous result applies to $F(\Cu(A))$.
Moreover, by \cite[Theorem~4.4]{EllRobSan11Cone}, the cone of functionals $F(\Cu(A))$ is isomorphic (as an ordered topological cone) to the cone of lower semicontinuous 2-quasitraces $\QT(A)$ on $A$ via the assignment
\[
\QT(A)\to F(\Cu(A)),\quad \tau\mapsto d_\tau.
\]
Here, $d_\tau([a]):=\lim_n \tau(a^{1/n})$ for $a\in (A\otimes\mathcal K)_+$.
We thus obtain the following result, which does not seem to have appeared in the literature before.

\begin{cor}
\label{prp:refinementQT}
Let $A$ be a \ca{}.
Then the cone $\QT(A)$ of lower semicontinuous $2$-quasitraces satisfies Riesz refinement.
\end{cor}

\begin{rmk}
\label{rmk:refinementQT}
For unital, simple \ca{s}, \autoref{prp:refinementQT} follows from more classical results of Blackadar and Handelman, \cite{BlaHan82DimFct}.
Indeed, they show that if~$A$ is a unital \ca{}, then the cone $\QT_b(A)$ of bounded 2-quasitraces is lattice ordered.
Since this cone embeds in a vector space, it follows from the well known equivalence between Riesz interpolation and Riesz refinement in the setting of ordered abelian groups that $\QT_b(A)$ has Riesz refinement;
see, for example, \cite[Proposition~2.1]{Goo86GpsInterpolation}.
If $A$ is also simple, then $\QT(A)=\QT_b(A)\cup \{\tau_\infty\}$, where $\tau_\infty\colon A_+\to [0,\infty]$ is infinite on all non-zero elements of $A_+$.
It is then straightforward to extend the Riesz refinement from $\QT_b(A)$ to $\QT(A)$.
\end{rmk}

Given a \CuSgp{} $S$ satisfying \axiomO{5}, recall that $F(S)'$ denotes the family of linear, lower semicontinuous functions $f\colon F(S)\to [0,\infty]$.
(Note that $F(S)'$ is denoted by $\text{Lsc}(F(S))$ in \cite{AntPerThi18TensorProdCu} and \cite{Rob13Cone}.)
Given $x\in S$, we obtain $\widehat{x}\in F(S)'$ defined by $\widehat{x}(\lambda)=\lambda(x)$, for $\lambda\in F(S)$.
Since $F(S)$ is an algebraically ordered, lattice ordered, compact cone, we may apply
\autoref{prp:coneDualRealizeInfSup} and~\autoref{prp:coneLscDual} to obtain:

\begin{prp}
\label{prp:FSLscDual}
Let $S$ be a \CuSgp{} satisfying \axiomO{5} and \axiomO{6}.
Then $F(S)'$ is an inf-semilattice-ordered, directed complete cone, with infimum given as in \eqref{prp:MdualInj:eqInfDual}.
In particular, given $x,y\in S$, the infimum of $\widehat{x}$ and $\widehat{y}$ in $F(S)'$ satisfies
\begin{align}
\label{prp:FSLscDual:eqInf}
(\widehat{x}\wedge\widehat{y})(\lambda) 
= \inf \big\{ \lambda_1(x) + \lambda_2(y) : \lambda=\lambda_1+\lambda_2 \big\},
\end{align}
for all $\lambda\in F(S)$, and the infimum is attained.
\end{prp}

\begin{qst}
\label{qst:FSLscDual}
Let $S$ be a countably-based \CuSgp{} satisfying \axiomO{5} and \axiomO{6}.
There is a natural semigroup morphism $\widehat{} \colon S\to F(S)'$ given by $x\mapsto\widehat{x}$.
By \autoref{prp:FSLscDual}, $F(S)'$ is inf-semilattice ordered.
Thus, if $S$ is also inf-semilattice ordered, it is natural to ask wether $\widehat{x\wedge y} = \widehat{x}\wedge\widehat{y}$, for all $x,y\in S$.
This question will be taken up in \autoref{sec:Edwards}.
\end{qst}

\subsection{Well capped cones}
\label{subsec:wellcapped}
Recall that a subset $K$ of a topological cone $C$ is called a \emph{cap} if it is compact, convex, and $C\backslash K$ is also convex.
The cone $C$ is said to be \emph{well capped} if it is the union of its caps;
see, for example, \cite{Phe01LNMChoquet}.

Let $S$ be a \CuSgp{} satisfying \axiomO{5}.
In this subsection we show that the cone $F(S)$ contains many well capped subcones;
see \autoref{lma:wellcapped}.
If $S$ is also countably based, then $F(S)$ naturally decomposes as the disjoint union of well capped, cancellative subcones.

Recall that $L(F(S))$ is defined as a certain subset of $F(S)'$, which can be identified with the sequential closure of the span of the set $\{t\widehat{x} : t\in(0,\infty),x\in S\}$ in $F(S)'$;
see \cite{Rob13Cone}.
Further, we have $L(F(S))\cong S\otimes[0,\infty]$;
see \cite[Section~7.5, p.132ff]{AntPerThi18TensorProdCu}.

Given an ideal $J$ in $S$, we let $\lambda_J\in F(S)$ denote the functional that is $0$ on $J$ and $\infty$ otherwise.
Then $2\lambda_J=\lambda_J$.
Moreover, every idempotent in $F(S)$ arises this way for some ideal, that is, $E(F(S))$ with the reverse order is naturally order-isomorphic to the lattice of ideals in $S$.

By a subcone of $F(S)$ we understand a subset that is closed under addition and multiplication by strictly positive scalars. 
Given an ideal $J$ in $S$, we set
\[
F_J(S) := \lambda_J + \big\{ \lambda\in F(S) : \lambda(x')<\infty\hbox{ whenever } x'\ll x \text{ for some } x\in J  \big\}.
\]
Then $F_J(S)$ is a subcone of $F(S)$ with apex $\lambda_J$. 
The cone $F(S)$ decomposes as the disjoint union of the subcones $F_J(S)$, with $J$ ranging over the ideals of $S$.
The \emph{support ideal} of $\lambda\in F(S)$ is the unique ideal $J$ such that $\lambda\in F_J(S)$.
One can show that the support ideal of $\lambda$ is $J$ if and only if $\varepsilon(\lambda)=\lambda_J$.

\begin{prp}
\label{lma:wellcapped}
Let $S$ be a \CuSgp{} satisfying \axiomO{5}, and let $J$ be a countably generated ideal of $S$.
Then $F_J(S)$ is well capped.
\end{prp}
\begin{proof}
Since $J$ is countably generated, it contains a largest element;
see \cite[Paragraph~5.1.6, p.39f]{AntPerThi18TensorProdCu}, and also the comments before \autoref{prp:infWithIdeal}.
Choose a $\ll$-increasing sequence $(x_n)_n$ whose supremum is the largest element of $J$.
Let $\lambda\in F_J(S)$.
Then $(\lambda(x_n))_n$ is an increasing sequence in $[0,\infty)$.
Define 
\[
f = \sum_{n=1}^\infty \alpha_n\widehat{x_n} \in L(F(S)),
\]
where we choose the numbers $(\alpha_n)_n$ in $(0,\infty)$ such that $\alpha_n\to 0$ fast enough so that $f(\lambda)\leq 1$.
Observe that $\widehat{x}\leq \infty f$ for any $x\in J$.

We consider
\[
C_f := \big\{ \mu\in F_J(S) : f(\mu)\leq 1 \big\},
\]
which contains $\lambda$.
Let us show that $C_f$ is a cap of $F_J(S)$.
Since $f$ is linear, both $C_f$ and its complement in $F_J(S)$ are convex.
It remains to show that $C_f$ is compact.

We show first that if $\mu\in F(S)$ is such that $f(\mu)\leq 1$ then $\lambda_J+\mu\in F_J(S)$.
Let $x'\ll x$ in $J$.
Using \cite[Lemma~2.2.5]{Rob13Cone} at the first step, we get
\[
\widehat{x'} \ll 2\widehat{x} \leq \infty f.
\]
Hence, $\widehat{x'}\leq Nf$ for some $N\in\NN$.
Then $\mu(x')\leq N<\infty$, which in turn implies that $\lambda_J+\mu\in F_J(S)$.
Thus, $C_f$ agrees with $\lambda_J+\{\mu\in F(S):f(\mu)\leq 1\}$.
This set is closed in $F(S)$ and therefore compact.
\end{proof}

\section{Edwards' condition for abstract Cuntz semigroups}
\label{sec:Edwards}

In this section we introduce Edwards' condition for \CuSgp{s}; see \autoref{dfn:Edwards}.
This condition is inspired by a property considered by Edwards \cite[Condition~(2)]{Edw69UniformApproxAff}, and it has been studied in a more restrictive setting in \cite{Thi17arX:RksOps}.
In \autoref{prp:CuA-Edwards} below we show that Edwards' condition is satisfied by Cuntz semigroups of general \ca{s}.

\begin{dfn}
\label{dfn:Edwards}
Let $S$ be a \CuSgp{} and let $\lambda\in F(S)$.
We say that $S$ satisfies \emph{Edwards' condition for $\lambda$} if
\begin{align}
\label{dfn:Edwards:eq}
\inf\big\{ \lambda_1(x)+\lambda_2(y) \colon \lambda=\lambda_1+\lambda_2 \big\}
= \sup \big\{ \lambda(z) \colon z\leq x,y \big\},
\end{align}
for all $x,y\in S$.
If this holds for all $\lambda\in F(S)$, then we say that $S$ satisfies \emph{Edwards' condition}.
\end{dfn}

\begin{rmk}
\label{rmk:Edwards}
Let $S$ be a \CuSgp{} satisfying \axiomO{5} and \axiomO{6}.
It follows from \autoref{prp:coneDualRealizeInfSup} that the infimum in \eqref{dfn:Edwards:eq} is attained (see also the remarks before \autoref{prp:refinementFS}).
Further, it follows from \autoref{prp:FSLscDual} that the left hand side in \eqref{dfn:Edwards:eq} agrees with $(\widehat{x}\wedge\widehat{y})(\lambda)$.
Thus, $S$ satisfies Edwards' condition for $\lambda$ if and only if
\begin{align}
\label{dfn:equivEdwards}
(\widehat{x}\wedge\widehat{y})(\lambda)
= \sup \big\{ \lambda(z) :  z\leq x,y \big\},
\end{align}
for all $x,y\in S$.
Notice that the inequality `$\geq$' always holds.

If $S$ is also an inf-semilattice, then we have $\sup\{\lambda(z) : z\leq x,y\}=\widehat{x\wedge y}(\lambda)$.
Therefore, in this setting, Edwards' condition is equivalent to
\[
\widehat{x}\wedge\widehat{y}=\widehat{x\wedge y},
\]
for all $x,y\in S$. (See \autoref{qst:FSLscDual}.)
\end{rmk}

Next, we show that the supremum in \eqref{dfn:Edwards:eq} is achieved.
We first need a lemma.

\begin{lma}
\label{prp:improveEC}
Let $S$ be a \CuSgp{} satisfying \axiomO{5} and \axiomO{6}, let $\lambda\in F(S)$ such that $S$ satisfies Edwards' condition for $\lambda$, let $z'\ll z\leq x,y$ in $S$, and let $t\in\RR$ satisfy $t<(\widehat{x}\wedge\widehat{y})(\lambda)$.
Then there exists $\tilde{z}\in S$ such that $z'\ll \tilde{z}\leq x,y$ and $t<\lambda(\tilde{z})$.
\end{lma}
\begin{proof}
If $t<\lambda(z)$, then we can set $\tilde{z}:=z$.
Thus, we may assume that $\lambda(z)\leq t$, and in particular $\lambda(z)$ is finite.
We distinguish two cases.

Case~1.
Assume that $(\widehat{x}\wedge\widehat{y})(\lambda)<\infty$. 
In this case, choose $\varepsilon>0$ such that $t+\varepsilon< (\widehat{x}\wedge\widehat{y})(\lambda)$.
Since $\lambda(z)<\infty$, we can choose $z''\in S$ such that
\[
z'\ll z''\ll z, \andSep 
\lambda(z) < \lambda(z'')+\varepsilon/2.
\]
Applying \axiomO{5} to $z''\ll z\leq x$ and $z''\ll z\leq y$, we obtain $u,v\in S$ such that
\begin{align*}
z''+u \leq  x\leq z+u, \andSep
z''+v \leq  y\leq  z+v.
\end{align*}
Since $(\widehat{u}\wedge\widehat{v})(\lambda)<\infty$, we can apply Edwards' condition to obtain $w\in S$ such that
\[
w\leq u,v, \andSep
(\widehat{u}\wedge\widehat{v})(\lambda) \leq \lambda(w)+\varepsilon/2.
\]	
Set $\tilde{z} := z''+w$.
Then $z'\ll \tilde z\leq x,y$.
Using that $F(S)'$ is semilattice-ordered (\autoref{prp:FSLscDual}) at the second step, we deduce
\begin{align*}
(\widehat{x}\wedge\widehat{y})(\lambda)
&\leq (\widehat{z+u}\wedge \widehat{z+v})(\lambda) \\
&= \widehat{z}(\lambda) + (\widehat{u}\wedge\widehat{v})(\lambda) \\
&\leq \lambda(z'')+\varepsilon/2+\lambda(w)+\varepsilon/2 \\
&= \lambda(\tilde{z}) + \varepsilon,
\end{align*}
which implies $t<\lambda(\tilde{z})$.
This proves this case of the lemma.

Case~2.
Suppose that $(\widehat{x}\wedge\widehat{y})(\lambda)=\infty$.
Choose $z''\in S$ satisfying $z'\ll z''\ll z$.
Construct $u$ and $v$ as in case~1.
Then
\[
\infty 
= (\widehat{x}\wedge\widehat{y})(\lambda)
\leq (\widehat{z+u}\wedge \widehat{z+v})(\lambda) \\
= \widehat{z}(\lambda) + (\widehat{u}\wedge\widehat{v})(\lambda),
\]
which implies that $(\widehat{u}\wedge\widehat{v})(\lambda)=\infty$.
Applying Edwards' condition, we obtain $w\in S$ such that $w\leq u,v$ and $t<\lambda(w)$.
Then, as in Step~1, the element $\tilde{z} := z''+w$ has the desired properties.
\end{proof}

\begin{thm}
Let $S$ be a \CuSgp{} satisfying \axiomO{5} and \axiomO{6}, let $\lambda\in F(S)$ such that $S$ satisfies Edwards' condition for $\lambda$, and let $x,y\in S$. 
Then there exists $z\in S$ such that $z\leq x,y$ and
\[
(\widehat{x}\wedge\widehat{y})(\lambda)
= \lambda(z).
\]
Moreover, given also $z_0',z_0\in S$ with $z_0'\ll z_0\leq x,y$, the element $z$ may be chosen such that $z_0'\ll z$.
\end{thm}
\begin{proof}
Let $(t_n)_n$ be a strictly increasing sequence in $\RR$ with $\sup_n t_n = (\widehat{x}\wedge\widehat{y})(\lambda)$.
If $z_0'$ and $z_0$ are not given, we simply consider $z_0'=0$ and $z_0=0$.
We inductively construct $z_n',z_n\in S$ for $n\geq 1$ such that
\[
z_{n-1}' \ll z_n'\ll z_n \leq x,y, \andSep
t_n < \lambda(z_n'),
\]
for $n\geq 1$.

Given $n\geq 1$, assume that $z_{n-1}',z_{n-1}$ with $z_{n-1}'\ll z_{n-1}\leq x,y$ have been chosen.
Using \autoref{prp:improveEC}, we obtain $z_n\in S$ such that
\[
z_{n-1}'\ll z_n\leq x,y, \andSep
t_n< \lambda(z_n).
\]
Choose $z_n'\in S$ such that
\[
z_{n-1}'\ll z_n'\ll z_n, \andSep t_n < \lambda(z_n').
\]
Then $z_n'$ and $z_n$ have the claimed properties.

We obtain a $\ll$-increasing sequence $(z_n')_n$, which allows us to set $z:=\sup_n z_n'$.
Then $z_0'\ll z\leq x,y$ and
\[
(\widehat{x}\wedge\widehat{y})(\lambda)
= \sup_n t_n 
\leq \sup_n \lambda(z_n')
= \lambda(z),
\]
which implies that $z$ has the desired properties.
\end{proof}

\begin{cor}
\label{prp:exactEC}
Let $S$ be a \CuSgp{} satisfying \axiomO{5} and \axiomO{6} and Edwards' condition.
Then for every $\lambda\in F(S)$ and $x,y\in S$, there exist $\lambda_1,\lambda_2\in F(S)$ and $z\in S$ such that
\[
\lambda = \lambda_1+\lambda_2,\qquad
z \leq x,y, \andSep
\lambda_1(x)+\lambda_2(y) = \lambda(z).
\]
\end{cor}

The following result can be interpreted as the fact that the Edwards' condition implies that its dual version is also satisfied.
It is not clear wether or not these  conditions are actually equivalent.

\begin{prp}
\label{prp:dualEdwarda}
Let $S$ be a \CuSgp{} satisfying \axiomO{5} and \axiomO{6}.
Let $\lambda\in F(S)$ be such that $S$ satisfies Edwards' condition for $\lambda$.
Then
\begin{align}
\label{prp:dualEdwarda:eq1}
\sup \big\{ \lambda_1(x)+\lambda_2(y) : \lambda_1+\lambda_2=\lambda \big\} 
= \inf\big\{ \lambda(a) \colon x,y\leq a \big\},
\end{align}
for all $x,y\in S$.
\end{prp}
\begin{proof}
The inequality `$\leq$' in \eqref{prp:dualEdwarda:eq1} is straightforward to obtain. Let us show the opposite inequality.
Let $r$ denote the left side.
If $r=\infty$, we are done.
Let us thus suppose that $r<\infty$.
Observe that this implies that $\lambda(x)<\infty$ and $\lambda(y)<\infty$.

Applying \autoref{prp:exactEC}, we obtain $\lambda_1,\lambda_2\in F(S)$ and $z\in S$ such that
\[
\lambda = \lambda_1+\lambda_2, \andSep
z \leq x,y, \andSep
\lambda_1(x)+\lambda_2(y) = \lambda(z).
\]

Let $\varepsilon>0$.
Since $\lambda(z)$ is finite, we can choose $z'\ll z$ such that $\lambda(z)\leq\lambda(z')+\varepsilon$.
Applying \axiomO{5} for $z'\ll z\leq x$ and $z'\ll z\leq y$, we obtain $u,v\in S$ such that
\[
z'+u \leq x \leq z+u, \andSep 
z'+v \leq y \leq z+v.
\]
Set $a:=z+u+v$ which clearly satisfies $x,y\leq a$.
Then
\begin{align*}
\lambda(a)+\lambda_1(x)+\lambda_2(y)
&= \lambda(a)+\lambda(z) \\
&= \lambda(z)+\lambda(u)+\lambda(z)+\lambda(v) \\
& \leq \lambda(z')+\lambda(u)+\lambda(z')+\lambda(v)+2\varepsilon \\
& \leq \lambda(x)+\lambda(y)+2\varepsilon \\
&= \lambda_1(x)+\lambda_2(x)+\lambda_1(y)+\lambda_2(y)+2\varepsilon.
\end{align*}
Since $\lambda_1(x)$ and $\lambda_2(y)$ are finite, we may cancel them and obtain
\[
\lambda(a)
\leq \lambda_2(x)+\lambda_1(y)+2\epsilon\leq r+2\varepsilon,
\]
which implies the desired inequality.
\end{proof}

The following result shows that to prove Edwards' condition for $S$, it suffices to deal with the case where the functional $\lambda$ is finite on the given elements $x,y\in S$.
This reduction will come in handy when we prove Edwards' condition for Cuntz semigroups of C*-algebras.

\begin{thm}
\label{prp:EC-reduction}
Let $S$ be a \CuSgp{} satisfying \axiomO{5}, \axiomO{6} and \axiomO{7}, and let $\lambda\in F(S)$. Then $S$ satisfies Edwards' condition for $\lambda$ if, for all $x,y\in S$ with $\lambda(x),\lambda(y)<\infty$, we have
\[
(\widehat{x}\wedge\widehat{y})(\lambda)
\leq \sup \big\{ \lambda(z) : z\leq x,y \big\}.
\]
\end{thm}

We will prove the theorem using a series of lemmas.

Let $S$ be a \CuSgp{} and let $J\subseteq S$ be an ideal.
We define $\lambda_J\colon S\to[0,\infty]$ as in \autoref{subsec:wellcapped}, and $h_J\colon F(S)\to[0,\infty]$ as follows:
\begin{align*}
h_J(\lambda) = \begin{cases}
	0 & \text{if }\lambda\leq\lambda_J \\
	\infty & \text{if }\lambda\nleq\lambda_J
\end{cases}.
\end{align*}
Observe that, if $J$ has a largest element $w_J$ (for example, if $J$ is countably based), then $h_J=\widehat{w_J}$.

%
%
%

\begin{lma}
\label{prp:IdealEpsLambda}
Let $S$ be a \CuSgp{} and let $\lambda\in F(S)$.
Set
\[
J := \big\{ x\in S \colon \lambda(x')<\infty \text{ for all } x'\ll x \big\}.
\]
Then $J$ is an ideal in $S$ and $\varepsilon(\lambda)=\lambda_J$.
\end{lma}
\begin{proof}
It is straightforward to verify that $J$ is an ideal.

The sequence $(\tfrac{1}{n}\lambda)_n$ converges to $\varepsilon(\lambda)$ in $F(S)$.
By definition of the topology in $F(S)$, this means that for all $x',x\in S$ with $x'\ll x$, we have
\[
\limsup_n \tfrac{1}{n}\lambda(x')
\leq \varepsilon(\lambda)(x)
\leq \liminf_n \tfrac{1}{n}\lambda(x).
\]

To show that $\epsilon(\lambda)\geq \lambda_J$, let $x\in S$ satisfy $\epsilon(\lambda)(x)=0$.
We need to verify that $x\in J$.
Let $x'\ll x$.
If $\lambda(x')=\infty$, then $\limsup_n \tfrac{1}{n}\lambda(x') = \infty$, which contradicts
\[
\limsup_n \tfrac{1}{n}\lambda(x')
\leq \varepsilon(\lambda)(x)=0.
\] 
Thus, $\lambda(x')<\infty$.
Since this holds for all $x'\ll x$, we conclude that $x\in J$.

To show the converse inequality, let $x\in J$.
We need to verify that $\varepsilon(\lambda)(x)=0$.
Choose a $\ll$-increasing sequence $(x_n)_n$ with supremum $x$.
By assumption, we have $\lambda(x_n)<\infty$ for each $n$.
This implies that $\varepsilon(\lambda)(x_n)=0$.
Using that $\varepsilon(\lambda)$ preserves suprema of increasing sequences, we deduce that $\varepsilon(\lambda)(x)=0$.
\end{proof}

\begin{lma}
\label{prp:hatWedgeIdl}
Let $S$ be a \CuSgp{} satisfying \axiomO{5} and \axiomO{6}, let $J\subseteq S$ be an ideal of $S$,
and let $x\in S$.
Then
\[
(\widehat{x}\wedge h_J)(\lambda)
= \sup \big\{ \widehat{z}(\lambda) \colon z\in J, z\leq x\big\},
\]
for all $\lambda\in F(S)$.
\end{lma}
\begin{proof}
Let $\lambda\in F(S)$.
Recall that $F(S)$ is a complete lattice.
This allows us to define
\[
\lambda^{(J)} := \inf \big\{ \mu\in F(S) \colon \lambda\leq\mu+\lambda_J \big\}.
\]
We have $\lambda^{(J)}\leq\lambda$ and $\lambda^{(J)}+\lambda_J = \lambda+\lambda_J$.
The result will follow by combining the following two claims.

\textit{Claim~1}:
Given $y\in S$, we have
\begin{align*}
\lambda^{(J)}(y) = \sup \big\{ \lambda(z) \colon z\in J, z\leq y \big\}.
\end{align*}

To prove the claim, let $z\in J$ satisfy $z\leq y$.
Then $\lambda_J(z)=0$ and therefore
\[
\lambda(z)
=\lambda(z)+\lambda_J(z)
=\lambda^{(J)}(z)+\lambda_J(z)
=\lambda^{(J)}(z)
\leq\lambda^{(J)}(y),
\]
which shows inequality `$\geq$'. 
To prove the converse inequality, one shows that the function $\mu\colon S\to[0,\infty]$ defined by
\[
\mu(y):=\sup \big\{ \lambda(z) :z\in J, z\leq y\big\},
\]
for $y\in S$, is a functional on $S$ satisfying $\lambda\leq \mu + \lambda_J$.
By definition of $\lambda^{(J)}$ we obtain $\lambda^{(J)}\leq\mu$.
This proves the claim.

\textit{Claim~2}:
We have $(\widehat{x}\wedge h_J)(\lambda) = \lambda^{(J)}(x)$.
Indeed, using \eqref{prp:MdualInj:eqInfDual} at the first step and \autoref{prp:refinementFS} at the third step, we deduce
\begin{align*}
(\widehat{x}\wedge h_J)(\lambda)
&= \inf\big\{ \lambda_1(x)+h_J(\lambda_2) \colon \lambda= \lambda_1+\lambda_2 \big\} \\
&= \inf\big\{ \lambda_1(x) \colon \lambda=\lambda_1+\lambda_2, \lambda_2\leq\lambda_J \big\} \\
&= \inf\big\{ \mu(x) \colon \lambda\leq\mu+\lambda_J \big\} \\
&= \lambda^{(J)}(x),
\end{align*}
which proves the claim.
\end{proof}

Given a \CuSgp{} $S$ and $x\in S$, we let $\langle x\rangle$ denote the ideal generated by $x$.

\begin{lma}
\label{prp:hatInfIntoIdls}
Let $S$ be a \CuSgp{} satisfying \axiomO{5}, \axiomO{6} and \axiomO{7}, let $x,y\in S$, and let $\lambda\in F(S)$.
Then there exist $r,s\in S$ such that
\[
r\leq x,\quad
s\leq y,\quad
r,s\in \langle x\rangle \cap \langle y\rangle, \andSep
(\widehat{x}\wedge\widehat{y})(\lambda)
= (\widehat{r}\wedge\widehat{s})(\lambda).
\]
\end{lma}
\begin{proof}
Set $D_x:= \{r : r\in\langle y\rangle, r\leq x\}$.
By \autoref{prp:infWithIdealDirected}, $D_x$ is upward directed.
By \autoref{prp:coneLscDual}, infima commute with directed suprema in $F(S)'$.
Using this at the last step, and using that $\widehat{y}\leq h_{\langle y\rangle}$ at the first step, and using \autoref{prp:hatWedgeIdl} at the third step, we obtain
\[
\widehat{x}\wedge\widehat{y}
= \widehat{x}\wedge\big(h_{\langle y\rangle}\wedge\widehat{y}\big)
= \big(\widehat{x}\wedge h_{\langle y\rangle}\big)\wedge\widehat{y}
= ( \sup_{r\in D_x} \widehat{r} ) \wedge \widehat{y}
= \sup_{r\in D_x} \big( \widehat{r}\wedge\widehat{y} \big).
\]
Similarly, the set $D_y:=\{s : s\leq y, s\in\langle x\rangle\}$ is upward directed.
We deduce
\[
\widehat{x}\wedge\widehat{y}
= \sup_{r\in D_x,s\in D_y} \big( \widehat{r}\wedge\widehat{s} \big).
\]
Choose sequences $(r_n)_n$ in $D_x$ and $(s_n)_n$ in $D_y$ such that
\[
(\widehat{x}\wedge\widehat{y})(\lambda)
= \sup_n \big( \widehat{r_n}\wedge\widehat{s_n} \big)(\lambda).
\]
Using that $D_x$ and $D_y$ are upward directed, we may assume that $(r_n)_n$ and $(s_n)_n$ are increasing.
Then $r:=\sup_n r_n$ and $s:=\sup_n s_n$ have the desired properties.
\end{proof}

\begin{lma}
Let $S$ be a \CuSgp{} satisfying \axiomO{6}, and let $x,y\in S$.
Then $\langle x\rangle \cap \langle y\rangle$ is the ideal generated by $\{z\in S : z\leq x,y\}$.
\end{lma}
\begin{proof}
Let $z\in \langle x\rangle \cap \langle y\rangle$ and $z'\ll z$.
Then $z'\leq nx$ and $z'\leq ny$ for some $n\in \NN$.
Let $z''\ll z'$.
By \axiomO{6} used in $z''\ll z'\leq nx$, there exist $x_1,\ldots,x_n$ such that $z''\ll \sum_{k=1}^n x_k$ and $x_k\leq z', x$ for all $k$.
Choose for each $k$ an element $x_k'\ll x_k$ such that $z''\ll \sum_{k=1}^n x_k'$.

For each $k$, applying \axiomO{6} again in $x_k'\ll x_k\leq ny$, we obtain $y_{k,1},\ldots,y_{k,n}$ such that $x_k'\ll \sum_{l=1}^n y_{k,l}$ and $y_{k,l}\leq x_k,y$ for all $l$.
It follows that $y_{k,l}$ belongs to the set $\{z \colon z\leq x,y\}$ for all $k$ and $l$.
Hence, $z''$ belongs to the ideal of $S$ generated by this set.
Since $z''$ and $z'$ can be chosen arbitrarily such that $z''\ll z'\ll z$, we deduce that $z$ belongs to this ideal as well. 
\end{proof}


\begin{proof}[Proof of \autoref{prp:EC-reduction}]
Let $x,y\in S$ and $\lambda\in F(S)$.
The inequality `$\geq$' in \eqref{dfn:equivEdwards} is clear (see also \eqref{dfn:Edwards:eq}).
We prove the opposite inequality, that is,
\[
(\widehat{x}\wedge\widehat{y})(\lambda)
\leq \sup \big\{\lambda(z) : z\leq x,y \big\}.
\]

If the right hand side is $\infty$ we are done.
Let us thus assume that $x$, $y$, and $\lambda$ are such that if $z\leq x,y$ then $\lambda(z)$ is finite.
Let $J\subseteq S$ be the ideal as in \autoref{prp:IdealEpsLambda} such that $\epsilon(\lambda)=\lambda_J$. Namely, $J=\{x\in S\colon \lambda(x')<\infty \text{ for all }x'\ll x\}$.
Then $z\in J$ whenever $z\leq x,y$.
Thus, by the previous lemma, we have $\langle x\rangle\cap \langle y\rangle \subseteq J$. 

Use \autoref{prp:hatInfIntoIdls} to obtain $r,s\in S$ such that
\[
r\leq x,\quad
s\leq y,\quad
r,s\in \langle x\rangle \cap \langle y\rangle\subseteq J,\andSep
(\widehat{x}\wedge\widehat{y})(\lambda)
= (\widehat{r}\wedge\widehat{s})(\lambda).
\]
Choose $\ll$-increasing sequences $(r_n)_n$ and $(s_n)_n$ with suprema $r$ and $s$, respectively.
Since $r,s\in J$, we have that $\lambda(r_n)<\infty$ and $\lambda(s_n)<\infty$ for each $n$.
By assumption, we can choose $z_n$ such that
\[
(\widehat{r_n}\wedge\widehat{s_n})(\lambda)-\tfrac{1}{n}\leq\lambda(z_n), \andSep 
z_n\leq r_n,s_n.
\]
Using \autoref{prp:coneLscDual} at the second step, we deduce
\[
(\widehat{x}\wedge\widehat{y})(\lambda)
= (\widehat{r}\wedge\widehat{s})(\lambda)
= \sup_n (\widehat{r_n}\wedge\widehat{s_n})(\lambda)
= \sup_n \lambda(z_n)
\leq \sup\big\{ \lambda(z) : z\leq x,y\big\},
\]
as desired.
\end{proof}


\section{Edwards' condition for Cuntz semigroups of C*-algebras}
\label{sec:ca}

In this section we prove the main result of this paper, namely that Cuntz semigroups of C*-algebras satisfy Edwards' condition.
To this end, we first recall necessary results and constructions from \cite{BlaHan82DimFct} and \cite{Haa14Quasitraces}. 

\begin{pgr}
\label{pgr:AW*}
Let $A$ be a C*-algebra, and let $\tau\colon A\to\CC$ be a bounded $2$-quasitrace on $A$. Denote by $\ell^\infty(A)$ the \ca{} of norm-bounded sequences in $A$.  

Given a free ultrafilter $\mathcal U$ on $\NN$, let $J_\tau \subseteq \ell^\infty(A)$ be defined as
\[
J_\tau = \big\{ (a_n)_n\in \ell^\infty(A) : \lim_{\mathcal U}\tau(a_n^*a_n)=0 \big\}.
\]
Then $J_\tau$ is a closed, two-sided ideal and $M_\tau:=\ell^\infty(A)/J_\tau$ is an \AW-algebra. Moreover, there exists a bounded $2$-quasitrace $\bar \tau\colon M_\tau\to \CC$ such that 
\[
\bar\tau\big( \pi((a_n)_n) \big)=\lim_{\mathcal U} \tau(a_n), 
\]
where $\pi\colon \ell^\infty(A)\to M_\tau$ denotes the quotient map.
(See \cite[Proposition~4.2]{Haa14Quasitraces} and \cite[I.4 and II.2]{BlaHan82DimFct}.)

The $2$-quasitrace $\tau$ extends to a lower semicontinuous $2$-quasitrace on $A\otimes\KK$.
As in \autoref{coneFS}, we denote by $d_\tau\in F(\Cu(A))$ the functional associated to $\tau$. Recall that $d_\tau([a]):=\lim_n \tau(a^{1/n})$ for all $a\in (A\otimes\mathcal K)_+$.
Since this is independent of the class $[a]$ of $a$, we may also write $d_\tau(a)$ in place of $d_\tau([a])$. 
Recall also that the assignment $\tau\mapsto d_\tau$ allows us to identify the cone of lower semicontinuous 2-quasitraces $\QT(A)$ with $F(\Cu(A))$;
see \cite{EllRobSan11Cone}.

Now, for $a\in A_+$, we set $p_a:=\pi((a^{1/n})_n)\in  M_\tau$. Then $p_a$ is a projection in $M_\tau$ such that $\bar{\tau}(p_a)=d_\tau(a)$.
\end{pgr}


\begin{lma}
\label{prp:infHatViaAW}
Let $A$ be a C*-algebra, let $\tau$ be a bounded 2-quasitrace on $A$, and let $a,b\in A_+$.
Then
\begin{align}
\label{prp:infHatViaAW:eqStatement}
(\widehat{[a]}\wedge\widehat{[b]})(d_\tau)
= \max\big\{ \bar{\tau}(q) \colon q\in M_\tau\text{ is a projection such that } q\precsim p_a,p_b \big\}.
\end{align}
\end{lma}
\begin{proof}
We identify $F(\Cu(A))$ with $\QT(A)$ as explained above.
Set $\lambda:=d_\tau$.
It follows from \autoref{prp:FSLscDual} that 
\[
(\widehat{[a]}\wedge\widehat{[b]})(\lambda)
= \inf \big\{ \lambda_1([a]) + \lambda_2([b]) : \lambda=\lambda_1+\lambda_2 \big\}.
\]

Let $\lambda_1,\lambda_2\in F(\Cu(A))$ satisfy $\lambda=\lambda_1+\lambda_2$.
Let $\tau_1,\tau_2\in\QT(A)$ be such that $\lambda_1=d_{\tau_1}$ and $\lambda_2=d_{\tau_2}$
Then $\tau=\tau_1+\tau_2$.
It follows that $\tau_1$ and $\tau_2$ are bounded $2$-quasitraces that induce bounded $2$-quasitraces $\bar{\tau}_1$ and $\bar{\tau}_2$ on $M_\tau$ such that $\bar{\tau}=\bar{\tau}_1+\bar{\tau}_2$.

Let $q\in M_\tau$ be a projection satisfying $q\precsim p_a,p_b$.
(For projections, Cuntz subequivalence as recalled at the beginning of \autoref{sec:Cu} agrees with Murray-von Neumann subequivalence.)
Then
\[
\lambda_1([a]) + \lambda_2([b])
= d_{\tau_1}(a)+d_{\tau_2}(b)
= \bar{\tau}_1(p_a)+\bar{\tau}_2(p_b)
\geq \bar{\tau}_1(q)+\bar{\tau}_2(q)
= \bar{\tau}(q).
\]
Passing to the infimum over all decompositions $\lambda=\lambda_1+\lambda_2$ and the supremum over all such projections $q$, we obtain the inequality~`$\geq$' in~\eqref{prp:infHatViaAW:eqStatement}.

Let us show the converse. By \cite[Corollary~14.1, p.80]{Ber72BearRgs}, \AW-algebras have \emph{generalized comparability}, that is, given two projections $e,f$ there exists a central projection $z$ such that $ze\precsim zf$ and $(1-z)e\succsim(1-z)f$.  Applied to $p_a,p_b\in M_\tau$, we obtain a central projection $z\in M_\tau$ such that $zp_a\precsim zp_b$ and $(1-z)p_a\succsim(1-z)p_b$.
Set
\[
r:= zp_a + (1-z)p_b.
\]
Then $r$ is a projection satisfying $r\precsim p_a,p_b$.
Given a projection $r'\in M_\tau$ with $r'\precsim p_a,p_b$ let us verify $r'\precsim r$.
Indeed, $r'\precsim p_a$ implies $zr'\precsim zp_a$ and similarly we obtain $(1-z)r'\precsim(1-z)p_b$.
Then
\[
r'
= zr' + (1-z)r'
\precsim zp_a + (1-z)p_b = r.
\]
Thus, the right hand side in \eqref{prp:infHatViaAW:eqStatement} is equal to $\bar{\tau}(r)$.
Define $\bar\tau_1,\bar\tau_2\colon M_\tau\to\CC$ by  
\[
\bar{\tau}_1(y)=\bar{\tau}(zy)\andSep \bar{\tau}_2(y)=\bar{\tau}((1-z)y)
\] for all $y\in M_\tau$.

Now regard $A$ embedded in $\ell^\infty(A)$ as constant sequences, and let $\tau_1,\tau_2\colon A\to \CC$
be the induced 2-quasitraces on $A$, that is, 
\[
\tau_1(a)=\bar \tau_1(\pi(a)) \andSep  \tau_2(a) = \bar{\tau}_2(\pi(a))
\]
for all $a\in A$. Then $\tau_1,\tau_2\in\QT(A)$ and $\tau=\tau_1+\tau_2$. Thus, $\lambda=d_{\tau_1}+d_{\tau_2}$.  It follows that
\begin{align*}
(\widehat{[a]}\wedge\widehat{[b]})(\lambda)
&= \inf \big\{ \lambda_1([a]) + \lambda_2([b]) : \lambda=\lambda_1+\lambda_2 \text{ in } F(\Cu(A)) \big\} \\
&\leq d_{\tau_1}([a]) + d_{\tau_2}([b])
= \bar{\tau}_1(p_a) + \bar{\tau}_2(p_b) \\
&= \bar{\tau}(zp_a)+\bar{\tau}((1-z)p_b)
= \bar{\tau}(zp_a+(1-z)p_b)
= \bar{\tau}(r),
\end{align*}
which completes the proof.
\end{proof}

\begin{thm}
\label{prp:CuA-Edwards}
Let $A$ be a C*-algebra.
Then $\Cu(A)$ satisfies Edwards' condition.
\end{thm}
\begin{proof}
First, we may assume that $A$ is stable.
Recall that $\Cu(A)$ is a \CuSgp{} satisfying \axiomO{5} and \axiomO{6}.
By \autoref{prp:O7} it also satisfies \axiomO{7}.
Hence by \autoref{prp:EC-reduction}, it is enough to show that
\[
(\widehat{[a]}\wedge\widehat{[b]})(\lambda)
\leq \sup \big\{ \lambda([c]) \colon [c]\leq [a],[b] \big\}
\]
for all $\lambda\in F(\Cu(A))$ and $[a],[b]\in \Cu(A)$ with $\lambda([a]),\lambda([b])<\infty$.
We continue to identify $F(\Cu(A))$ with $\QT(A)$, and therefore we consider $\tau\in \QT(A)$ and $a,b\in A_+$ with $d_\tau(a),d_\tau(b)<\infty$.

Let $h=a+b$.
Observe that $a,b\in\overline{hAh}$ and $d_\tau(h)<\infty$.
Set $B:=\overline{hAh}$.
The restriction of $\tau$ to $B$ is a bounded $2$-quasitrace with norm $d_\tau(h)$.
Choose a free ultrafilter $\mathcal U$ on $\NN$ and consider the \AW-algebra $M_\tau$, with bounded 2-quasitrace  $\bar{\tau}$, associated to the pair $(B,\tau)$ as described in Paragraph \ref{pgr:AW*}. Set $p_a=\pi((a^{1/n})_n)$ and $p_b=\pi((b^{1/n})_n)$ where $\pi$ is the quotient map.

Apply \autoref{prp:infHatViaAW} to obtain a projection $q\in M_\tau$ satisfying
\[
(\widehat{[a]}\wedge\widehat{[b]})(\lambda)=\bar{\tau}(q) \andSep
q\precsim p_a,p_b.
\]
We may assume that $q\leq p_a$. Choose $v\in M_\tau$ with $q=vv^*$ and $v^*v\leq p_b$.
Lift $v$ to a contractive element $\bar{v}=(v_n)_n$ in $\ell^\infty(B)$. For each $n$, set $w_n:=a^{1/n}v_nb^{1/n}$. Set $w:=(w_n)_n$. Then $w\in\ell^\infty(B)$ and $\pi(w)=p_avp_b=v$.

Let $t<(\widehat{a}\wedge\widehat{b})(\lambda)$. Since
\[
(\widehat{[a]}\wedge\widehat{[b]})(\lambda)=\bar{\tau}(q)=\lim_{\mathcal U}\tau(w_nw_n^*),
\]
there exists $n\in \NN$ such that $t<\tau(w_nw_n^*) \leq d_\tau(w_nw_n^*)$. Set $c:=w_nw_n^*$ which by construction satisfies $c\precsim a,b$.
Therefore $t\leq d_\tau(c)$ and the result follows. 
\end{proof}



\providecommand{\etalchar}[1]{$^{#1}$}
\providecommand{\bysame}{\leavevmode\hbox to3em{\hrulefill}\thinspace}
\providecommand{\noopsort}[1]{}
\providecommand{\mr}[1]{\href{http://www.ams.org/mathscinet-getitem?mr=#1}{MR~#1}}
\providecommand{\zbl}[1]{\href{http://www.zentralblatt-math.org/zmath/en/search/?q=an:#1}{Zbl~#1}}
\providecommand{\jfm}[1]{\href{http://www.emis.de/cgi-bin/JFM-item?#1}{JFM~#1}}
\providecommand{\arxiv}[1]{\href{http://www.arxiv.org/abs/#1}{arXiv~#1}}
\providecommand{\doi}[1]{\url{http://dx.doi.org/#1}}
\providecommand{\MR}{\relax\ifhmode\unskip\space\fi MR }
\providecommand{\MRhref}[2]{%
  \href{http://www.ams.org/mathscinet-getitem?mr=#1}{#2}
}
\providecommand{\href}[2]{#2}

\end{document}